\documentclass{amsart}
\usepackage{amsfonts}
\usepackage{amsmath}
\usepackage{amssymb}
\usepackage[utf8]{inputenc}
\usepackage{hyperref}
\usepackage{color}

\usepackage{lipsum}

\usepackage{graphicx}

\def\C{\mathbb{C}}
\def\R{\mathbb{R}}

\def\H{\mathbb{H}}
\def\Z{\mathbb{Z}}

\def\F{\mathcal{F}}

\def\F{\mathcal {F}}

\def\Isom{{\sf{Isom}}}

\def\Aff{\mathsf{Aff}}
\def\Sim{\mathsf{Sim}}

\def\G{\mathcal G}

\def\Aut{{\sf{Aut}}}

\def\GL{{\sf{GL}}}
\def\O{{\sf{O}}}
\def\SO{{\sf{SO}}}

\def\Mink{{\sf{Mink}} }

\newtheorem*{theorem*}{Theorem}
\newtheorem{theorem}{{Theorem}}[section]
\newtheorem{proposition}[theorem]{{Proposition}}
\newtheorem{isom.ext}[theorem]{{Trivial isometric extension}}
\newtheorem{definition}[theorem]{{Definition}}
\newtheorem{lemma}[theorem]{{Lemma}}
\newtheorem{corollary}[theorem]{{Corollary}}
\newtheorem{remark}[theorem]{{Remark}}
\newtheorem{example}[theorem]{{Example}}
\newtheorem{claim}[theorem]{{\sc Claim}}

\definecolor{purple}{rgb}{0.65,0.12,0.94}
\definecolor{forestgreen}{rgb}{0.4,0.64,0.13}

\hypersetup{
	colorlinks=true,
	linkcolor=red,
	filecolor=magenta,      
	urlcolor=black,
}

\title{On the Inoue--Bombieri construction}
\author [B. Flamencourt]{Brice Flamencourt}
\address{UMPA, CNRS, 
	\'Ecole Normale Sup\'erieure de Lyon, France}
\email{brice.flamencourt@ens-lyon.fr}

\author[A. Zeghib]{Abdelghani Zeghib }
\address{UMPA, CNRS, 
	\'Ecole Normale Sup\'erieure de Lyon, France}
\email{abdelghani.zeghib@ens-lyon.fr 
	\hfill\break\indent
	\url{http://www.umpa.ens-lyon.fr/~zeghib/}}


\begin{document}
\maketitle

\begin{abstract}
We study compact quotients of a Riemannian product $\R^q \times (N, g_N)$, where $(N, g_N)$ is a complete Riemannian manifold, by discrete subgroups $\Gamma$ of $\Sim(\R^q) \times \Isom(N)$. When $N$ is a symmetric space of non-compact type, this construction generalizes the well-known Inoue--Bombieri surfaces. We show that this setting is actually equivalent to that of the so-called LCP manifolds, and we establish a Bieberbach-type rigidity result in the case where $N$ is symmetric. In addition, we provide a classification of the manifolds $N$ and the groups $\Gamma$ when $N$ is a Hadamard manifold with strictly negative curvature.
\end{abstract}

\tableofcontents

\section{Introduction}
 
\subsection{GIB manifolds} Inoue surfaces were constructed in 1975 by Masahisa Inoue \cite{Inoue74} in his work on the classification of complex surfaces of Kodaira class VII, and independently by Enrico Bombieri. They are defined as quotients of the complex manifold $\C \times \H^2$ by discrete groups of automorphisms. One distinguishes three types of such surfaces: $S^-$, $S^0$, and $S^+$. In this work, we focus on surfaces of type $S^0$. These manifolds are compact quotients of $\C \times \H^2$ by discrete subgroups of $\Sim (\R^2) \times \Isom(\H^2)$ - identifying $\C$ with $\R^2$ -, where $\Sim(M)$ denotes the group of similarities of the Riemannian manifold $M$ (see equation~\eqref{Sim} below). They can be viewed as suspensions of a $3$-torus over the circle. We review their explicit construction in Example~\ref{S0 construction} below.

In this paper, we study a class of quotient manifolds that naturally generalizes the construction of $S^0$-Inoue--Bombieri surfaces. We call these spaces generalized Inoue--Bombieri (GIB) manifolds, and we define them as follows:

\begin{definition}   \label{GIB def}
A manifold $M$ is said to admit a generalized Inoue--Bombieri structure (in short, it is a GIB manifold) if it is a compact connected quotient of a product $\R^q \times (N,g_N)$ by a subgroup $\Gamma$ of $\Sim(\mathbb{R}^q) \times \Isom(N, g_N)$,  where $(N, g_N)$ is a Riemannian manifold, and $\Gamma$ is not contained in $\Isom(\R^q) \times \Isom(N, g_N)$.
\end{definition}

Since a GIB manifold $M$ is a quotient of a product $\R^q \times N$ by a discrete group $\Gamma$ preserving the factors, the foliation induced by the submersion $\R^q \times N \to N$ descends to a canonical foliation $\mathcal F$ on $M$. This foliation carries a natural transverse Riemannian structure (see \cite{Mol}), as $\Gamma$ acts by isometries on $N$. We will show the following structural theorem:

\begin{theorem} \label{Foliation}
Let $M$ be a GIB manifold and let $\mathcal F$ be its canonical foliation. Then, the closures of the leaves of $\mathcal F$ are finitely covered by flat tori. Up to replacing $N$ by its orthonormal frame bundle, we still get a GIB manifold and the closures of the leaves determine a fibration by tori over a compact manifold.
\end{theorem}

\subsection{LCP manifolds} It turns out that GIB manifolds are very similar to a recently studied class of manifolds, namely the locally conformally product (LCP) manifolds. These were introduced in \cite{FlaLCP} and defined as follows:

\begin{definition} [LCP manifolds] \label{DefLCP}
Let $M$ be a compact quotient of a Riemannian product $\tilde M := \R^q \times (N,g_N)$, where $q \ge 1$ and $(N, g_N)$ is simply connected, irreducible and non-complete, by a discrete subgroup $\Gamma$ of $\Sim(\tilde M) \cap (\Sim(\R^q) \times \Sim(N))$ not contained in $\Isom(\tilde M)$, acting freely and properly. The manifold $M := \Gamma \backslash \tilde M$ is called an LCP manifold.
\end{definition}

LCP manifolds arose from the study of torsion-free connections on compact conformal manifolds. In the situation where the connection preserves the conformal structure, it is called a {\em Weyl connection}. The Weyl connections that are locally but not globally the Levi-Civita connections of a metric in the conformal class are called closed, non-exact. It was believed in the early days of the theory that closed, non-exact Weyl connections on compact conformal manifolds were either flat or irreducible \cite{BeMo}. However, this conjecture was disproved by Matveev and Nikolayevsky \cite{MN, MN2}. The picture was later completed by Kourganoff \cite{Kou}, who showed that a third and last possibility could occur:

\begin{theorem*} [Kourganoff]
Let $M$ be a compact conformal manifold endowed with a closed, non-exact Weyl connection $\nabla$. If $\nabla$ is neither flat nor irreducible, then $M$ admits an LCP structure, and, using the notations of Definition~\ref{DefLCP}, the lift of $\nabla$ to $\tilde M$ is the Levi-Civita connection of $\tilde M$.
\end{theorem*}

A new synthetic proof of this result was recently given by the authors of the present paper in \cite{FlaZeg}, where the study of transversal similarity structures, playing a significant role in the analysis, was further developed.

It can be shown that $S^0$-Inoue--Bombieri surfaces admit an LCP structure. In light of this, one may ask more generally whether every GIB manifold admits an LCP structure. In this paper, we prove that these two classes actually coincide. More precisely we have:

\begin{theorem}[GIB-LCP equivalence] \label{Equivalence}
Let $M = \Gamma \backslash (\R^q \times N)$ be a GIB manifold. Then, there exists a function $f : N \to \R$ such that $\Gamma$ acts by similarities, not all isometries, on $\R^q \times (N, e^{2f} g_N)$. In particular, $M$ is an LCP manifold.

Conversely, if $M = \Gamma \backslash (\R^q \times N)$ is an LCP manifold. Then, there exists a function $f : N \to \R$ such that $\Gamma \subset \Sim(\R^q) \times\Isom (N, e^{2f} g_N)$. In particular, $M$ is a GIB manifold.
\end{theorem}

\subsection{OT manifolds} Inoue--Bombieri surfaces of type $S^0$ fall into the broader family of OT manifolds, constructed by Oeljeklaus and Toma in \cite{OT}. These are defined as compact quotients of $\C^t \times (\H^2)^s$ - for some positive integers $s$ and $t$ - by a discrete subgroup of $\mathrm{Aut} (\C)^t \times \mathrm{Aut}(\H^2)^s$.

Since it was proved in \cite{FlaLCP} that OT manifolds always admit an LCP structure, they also carry a GIB structure by Theorem~\ref{Equivalence}. However, we should emphasize that this GIB structure is not unique in many cases.

Moreover, OT manifolds provide an illustrative example of a more general - and arguably more natural - framework obtained by modifying Definition~\ref{GIB def} so that $\Gamma$ is a subgroup of $G := \Aff(\mathbb{R}^q) \times \Isom (N)$, where  $\Aff(\mathbb{R}^q)$  denotes the full affine group of $\mathbb{R}^q$. The group $G$ may be viewed as the automorphism group of the Levi-Civita connection on $\mathbb{R}^q \times N$, provided $N$ has no local flat factor in its de Rham decomposition. We believe that this setting would offers a natural direction in which to extend the results of this paper.

\subsection{Bieberbach rigidity in the symmetric case} With our new understanding of GIB manifolds, we would like to investigate how to construct explicit examples. In particular, how can we find admissible discrete groups $\Gamma$ in Definition~\ref{GIB def}? We will address this question in the specific case where $(N,g_N)$ is a symmetric space. In this setting, the product $\R^q \times (N,g_N)$ can be viewed as a homogeneous Riemannian manifold $G/H$, and the quotient is obtained via the left action of a discrete group $\Gamma \subset G$ on $G/H$. Note that the structure of $\Gamma$ depends strongly on the group $G$, that is, on how the manifold is realized as a homogeneous space. This problem of finding a suitable $\Gamma$ has already been considered in other particular cases, for example in \cite{KaTh}, where the authors studied the situation in which $G$ is reductive.

One way to construct a discrete cocompact group acting on $\R^q \times N$ is to look for a connected subgroup of $\Sim(\R^q) \times \Isom(N)$ acting properly and transitively on the product $\R^q \times N$ and to take a lattice in this group. This leads to the following question: are all admissible groups $\Gamma$ lattices in a connected Lie subgroup of $\Sim(\R^q) \times \Isom(N)$ acting properly? This property, which can be formulated on any homogeneous manifold, is known as {\em Bieberbach rigidity} \cite{FriGol}. We provide a more detailed discussion of this notion in Section~\ref{sectionBieberbach}. In our setting, we will prove:

\begin{theorem} [Bieberbach Rigidity] \label{BRigidity}
Let $(N, g_N)$ be a simply connected symmetric space and let $q > 0$ be an integer. Let $\Gamma$ be a discrete subgroup of $\Sim(\R^q) \times \Isom(N, g_N)$ acting properly, freely and cocompactly on $\R^q \times N$. Then, up to taking a finite index subgroup of $\Gamma$, there exists a connected subgroup $L$ of $\Sim(\R^q) \times \Isom(N, g_N)$ acting properly and transitively on $\R^q \times N$ such that $\Gamma$ is a lattice in $L$.
\end{theorem}

We believe that Theorem~\ref{BRigidity} should hold for a general homogeneous Riemannian space $N$, but proving this appears to require a completely different technique. In addition, we remark, continuing the discussion of OT manifolds, that the Bieberbach-type rigidity established in Theorem~\ref{BRigidity} is analogous to the fact that OT manifolds are solvmanifolds \cite[Section 6]{Kas}.

\subsection{Classification results} The main objective would be to achieve complete classification of GIB manifolds; that is to determine which manifolds $N$ and groups $\Gamma$ may arise in Definition~\ref{GIB def}. The Bieberbach rigidity established earlier provides a first substantial step toward this goal. The next step is to classify the possible factors $N$. In the last part of the paper, we undertake this classification in the case where $N$ is a Hadamard manifold of strictly negative curvature.

A first illustrative example is given by the hyperbolic space $\H^n$, for which we will describe the projection of $\Gamma$ onto $\Isom(\H^n)$. We then move on to the general case of Hadamard manifolds with strictly negative curvature, using the same approach as in the case of the hyperbolic space. In this framework, we show:

\begin{theorem}[Classification in the case of negative curvature] \label{Classification Negative}
Let $(N^n, g_N)$ be a complete Riemannian manifold of strictly negative curvature and let $q > 0$ be an integer. Assume that $\Gamma \backslash (\R^q \times N)$ is a GIB manifold. Then, one has an isometry
\begin{equation} \label{isometryN}
(N, g_N) \simeq (\R \times \R^{n-1}, dt^2 + \langle S_t \cdot , S_t \cdot \rangle)
\end{equation}
where $S_t \in \GL(\R^{n-1})$ for all $t \in \R$. Denoting by ${\bar P}$ the closure of the projection of $\Gamma$ onto $\Isom(N)$ and by ${\bar P}^0$ its identity connected component, ${\bar P}^0$ is exactly the set of translations of the factor $\R^{n-1}$ in the identification \eqref{isometryN}.

If moreover $N$ is a homogeneous space, then either it is the hyperbolic space $\H^n$ or $\Isom(N)$ fixes a unique point at infinity and there exists an endomorphism $A$ of $\R^{n-1}$ such that $S_t = e^{tA}$ in \eqref{isometryN}.
\end{theorem}

\subsection{Organization of the paper}

The present paper is divided into three parts. In Section~\ref{SectionQuotient} we establish equivalence between GIB and LCP manifolds by proving Theorem~\ref{Equivalence} and Theorem~\ref{Foliation}. In Section~\ref{sectionBieberbach} we investigate Bieberbach rigidity and we prove Theorem~\ref{BRigidity}. We finally turn to the case where $N$ has negative curvature in Section~\ref{Section3}, where we prove Theorem~\ref{Classification Negative}.

\section{Compact quotients of product manifolds by similarities} \label{SectionQuotient}

We recall that a similarity, or homothety, between two Riemannian manifolds $(M_1, g_1)$ and $(M_2, g_2)$ is a diffeomorphism $\phi : M_1 \to M_2$ such that
\begin{equation} \label{Sim}
\phi^*g_2 = \lambda^2 g_1 
\end{equation}
for some positive real number $\lambda$, called the ratio of the similarity. The set of similarities from a Riemannian manifold $(M, g)$ to itself is denoted by $\Sim(M, g)$. The set of isometries of $(M, g)$, i.e. the similarities with ratio $1$, is denoted by $\Isom(M, g)$. When no confusion is possible, we will omit the metric $g$ in these notations.

The identity connected component of a Lie group $G$ is denoted by $G^0$.

We consider the Riemannian product $\tilde M :=  \R^q \times (N, g_N)$ where $q \ge 1$, $(N, g_N)$ is a simply connected, complete Riemannian manifold. Let $\Gamma$ be a discrete subgroup of $\Sim(\R^q) \times \Isom(N, g_N)$ acting freely, properly and cocompactly on $\tilde M$. In particular, $M := \Gamma \backslash \tilde M$ is a GIB manifold in the sense of Definition~\ref{GIB def}. We denote by $P$ the projection of $\Gamma$ onto $\Isom(N)$, and by $\bar P$ and ${\bar P}^0$ respectively the closure of $P$ in $\Isom(N, g_N)$ with respect to the compact-open topology and its identity connected component.

We assume that the projection of $\Gamma$ onto $\Sim(\R^q)$ contains at least on strict similarity, i.e. a similarity with ratio different from $1$. Under this assumption, it is known that $P$ is isomorphic to $\pi_1(M)$; more precisely, the projection $\pi_1(M) \to P$ is injective, by the proof of \cite[Lemma 2.10]{FlaLCP}, which follows that of \cite[Lemma 4.17]{Kou}. We denote this isomorphism by $\varphi : P \to \pi_1(M)$. Let $\rho_0 : \Sim(\R^q) \to \R^*_+$ be the group homomorphism assigning to a similarity its ratio, and set $\rho := \rho_0 \circ \varphi$.

Our goal is to prove the following:

\begin{proposition} \label{mainProp}
There exists a smooth positive function $f : N \to \R$ which is $P$-equivariant, i.e. for any $p \in P$ one has $p^* f = \rho(p) f$.
\end{proposition}

We begin with a technical lemma:

\begin{lemma} \label{tech1}
There exists $\delta > 0$ such that for any $p \in P$ with $\rho (p) \neq 1$ and for any $x \in N$, one has $d_N (p(x), x) \ge \delta$, where $d_N$ stands for the Riemannian distance on $N$. 
\end{lemma}
\begin{proof}
The group $\Gamma$ preserves the product decomposition $\tilde M \simeq \R^q \times N$. Thus, the two canonical transverse foliations given by the product structure induce transverse foliations $\F$ and $\G$ on $M$. In addition, the group $\Gamma$ projects to isometries on the second factor $N$. Thus, the Riemannian exponential of $N$ descends to a map $\Xi : T \G \to M$, and the Riemannian metric $g_N$ descends to a Riemannian bundle metric on $T \G \to M$. By compactness, there exists $\delta > 0$ such that for any $y \in M$, the map $\Xi$ is injective on the open ball of radius $\delta$ of $T_y \G$ (it suffices to find a finite covering of $M$ by open subsets which are the projection of subsets of the form $B \times V$ with $B$ a ball of $\R^q$ and $V$ a small ball of $N$).

Now, let $p \in P$ with $\rho(p) \neq 1$. The associated map $\varphi(p)$ can be written as $(\phi, p) \in \Sim(\R^q) \times \Isom(N)$ and $\phi$ has a unique fixed point $a \in \R^q$. Let $x \in N$ and let $B_N(x, \delta)$ be the image of the open ball of radius $\delta$ in $T_x N$ by the Riemannian exponential map of $N$. We have $(\phi, p) (a,x) = (a,p(x))$. But the previous discussion implies that the restriction of the projection $\tilde M \to M$ to $\{a\} \times B_N(x, \delta)$ is injective. Consequently, $(a, p(x)) \notin \{a\} \times B_N(x, \delta)$ i.e. $p(x) \notin B_N(x, \delta)$ and $d_N(x, p(x)) \ge \delta$.
\end{proof}

From the previous technical lemma, we would like to deduce that the group homomorphism $\rho$ extends continuously to a group homomorphism $\bar P \to \R_+^*$. To achieve this, it suffices to establish the following result:

\begin{lemma} \label{tech2}
For any $p \in {\bar P}^0 \cap P$, one has $\rho(p) = 1$.
\end{lemma}
\begin{proof}
By Lemma~\ref{tech1}, there exists an open neighbourhood $U$ of $\mathrm{id}$ in $\Isom(N)$ such that $\rho(U \cap P) = \{1\}$. By replacing $U$ with $U \cap U^{-1}$ if necessary, we can assume that $U$ is symmetric, i.e. $U = U^{-1}$. It follows that for any $p \in P$, $\rho((p \cdot U) \cap P) = \{\rho(p)\}$.

We define $E$ as the set of all $p_0 \in {\bar P}^0$ such that there is a neighbourhood $V$ of $p_0$ in ${\bar P}^0$ satisfying $\rho(V \cap P) = \{1\}$. The set $E$ is open by definition and non-empty because $\mathrm{id} \in E$. We claim that $E$ is also closed. Indeed, if $p_0 \in {\bar P}^0$ is not in $E$, then there exists $p \in p_0 \cdot U$ such that $\rho(p) \neq 1$, and $p_0 \in p \cdot U$ because $U$ is symmetric. But $\rho((p \cdot U) \cap {\bar P}^0) = \{\rho (p)\}$, which means that for any $p_0' \in p\cdot U$ and any neighbourhood $V$ of $p_0'$, $P \cap V \cap (p\cdot U) \neq \emptyset$ because $P$ is dense in ${\bar P}^0$, and for any $p' \in P \cap V \cap (p\cdot U)$, $\rho(p') = \rho(p) \neq 1$, so $p_0'$ is not in $E$.

Since the set $E$ is non-empty, open and closed in the connected set ${\bar P}^0$, we conclude that $E = {\bar P}^0$. The lemma follows.
\end{proof}

\begin{corollary}
The group homomorphism $\rho : P \to \R^*_+$ extends uniquely to a continuous group homomorphism $\tilde \rho : \bar P \to \R^*_+$.
\end{corollary}
\begin{proof}
This is a direct consequence of Lemma~\ref{tech2}.
\end{proof}

We can now prove Proposition~\ref{mainProp}:

\begin{proof}[Proof of Proposition~\ref{mainProp}]
By assumption, the group $P$ acts cocompactly on $N$. Let $K$ be a compact subset of $N$ such that $\bar P \cdot K = N$ and let $f_0$ be a non-negative function with compact support such that $f_0 \vert_K = 1$. Let $\mu$ be the Haar-measure of $\bar P$ and we define for any $x \in N$:
\begin{equation}
f(x) := \int_{\bar P} \tilde \rho(p)^{-1} (p^* f_0)(x) d \mu(p).
\end{equation}
The function $f$ is well-defined because $\bar P$ acts properly on $N$ and it is positive because for any $x \in N$ there exists $p \in \bar P$ such that $p (x) \in K$, so $(p^* f_0) (x) = 1$. Moreover, the function is smooth by construction. It remains to prove that it is equivariant. One has for any $p' \in P$:
\begin{align*}
(p'^*f) (x) &= p'^* \int_{\bar P} \tilde \rho(p)^{-1} (p^* f_0)(x) d \mu(p) = \int_{\bar P} \tilde \rho(p)^{-1} (p'^*p^* f_0)(x) d \mu(p) \\
&= \int_{\bar P} \tilde \rho(p p')^{-1} \tilde \rho(p') ((p p')^* f_0)(x) d \mu(p) \\
&= \tilde \rho(p') \int_{\bar P} \tilde \rho(p)^{-1} \tilde \rho(p') (p^* f_0)(x) d \mu(p) \\
&= \rho(p') f(x). \qedhere
\end{align*}
\end{proof}

\begin{corollary} \label{corMetric}
There exists a metric $g_N'$ on $N$ such that $\Gamma \subset \Sim(\tilde M) \cap(\Sim(\R^q) \times \Sim(N, g'_N))$.
\end{corollary}
\begin{proof}
Taking $f$ to be the function given by Proposition~\ref{mainProp}, we define $g_N' := f^2 g_N$, and we easily verify that this metric has the desired property.
\end{proof}

This last corollary means that we can equivalently see $\Gamma$ as a subgroup of similarities of $(\tilde M, g_{\R^q} + g_N')$ preserving the decomposition $\R^q \times N$, acting freely, properly discontinuously and cocompactly on $\tilde M$ and containing a non-isometric similarity. The two settings are actually equivalent, since in the latter case, there exists an equivariant function on $N$ as shown in \cite[Proposition 3.6]{FlaLCP} and in a simpler way in  \cite{MP24}. Note that in these papers it is assumed that $N$ is an irreducible incomplete Riemannian manifold. However, the irreducibility is not relevant for the results we use here. In particular, the analysis done in \cite{Kou} applies and we have:

\begin{corollary} \label{P0abelian}
The group ${\bar P}^0$ is abelian.
\end{corollary}
\begin{proof}
This is a consequence of the previous discussion and of \cite[Lemma 4.1]{Kou}.
\end{proof}

The combination of Corollary~\ref{corMetric} and Corollary~\ref{P0abelian} implies Theorem~\ref{Equivalence}.

We now consider the foliation induced by the submersion $\tilde M \simeq \R^q \times N \to N$. Noticing that $\Gamma$ preserves the decomposition $\R^q \times N$, this foliation descends to a foliation $\F$ on $M = \Gamma \backslash \tilde M$. Applying \cite[Theorem 1.10]{Kou}, the closure of the leaves of $\F$ are finitely covered by flat tori. More precisely, by \cite[Lemma 4.18]{Kou} the subgroup $\Gamma_0$ of $\Gamma$ defined by $\Gamma_0 := \Gamma \cap (\Sim(\R^q) \times {\bar P}^0)$ is abelian and it is a lattice in $\R^q \times {\bar P}^0$. This proves the first part of Theorem~ref{Foliation}. The second part of the theorem comes from classical foliation theory: passing to the orthonormal frame bundle, the group ${\bar P}^0$ acts freely, thus making the foliation a fibration.

\section{Bieberbach rigidity} \label{sectionBieberbach}

\subsection{Bieberbach rigidity for homogeneous spaces} Bieberbach rigidity is a property of certain homogeneous spaces that greatly facilitates the search for discrete cocompact groups of diffeomorphisms.

We work within the following general framework. Let $X = G/H $ be a homogeneous space, and consider a discrete subgroup $ \Gamma \subset G$ acting (on the left) properly and cocompactly on $X$. This means that the quotient $\Gamma \backslash X$ is a compact Hausdorff space.

The situation is well understood when $X$ is of Riemannian type, meaning that the isotropy subgroup $H$ is compact. Equivalently, in this case, the full group $G$ acts properly on $X$. Then, the condition on $\Gamma $ reduces to requiring that $\Gamma$ is a cocompact lattice in $G$.

Our focus, however, is on non-Riemannian homogeneous spaces, i.e., those with non-compact isotropy $H$. In this setting, $\Gamma$ cannot be a cocompact lattice in $G$; otherwise, $G$ itself would act properly on $X$, which would contradict the fact that $X$ is non-Riemannian. Thus, $\Gamma$ must be thinner than a lattice, yet large enough to produce a compact quotient.

A particularly interesting situation arises when every such discrete group $\Gamma$ is contained in a connected closed subgroup $L \subset G$ acting properly on $X$, so that $\Gamma$ is in fact a lattice in   $ L$. One may view $L$ as a connected hull of $\Gamma$, reminiscent in many cases of its Zariski closure. The guiding philosophy is that studying connected Lie subgroups $L \subset G$ acting properly and cocompactly - a problem of linear-algebraic nature - is generally far more tractable than analyzing discrete subgroups $\Gamma$ directly - which has a number-theoretic flavor.

We define Bieberbach rigidity following the terminology of Fried and Goldman \cite{FriGol}:

\begin{definition} The homogeneous space  $X = G/H$ satisfies Bieberbach rigidity if every discrete subgroup $\Gamma$ of $G$ acting properly and cocompactly on $ X$ is contained in a connected closed subgroup $L\subset G$ acting properly on $ X$. 

\end{definition}

This form of rigidity was established in  \cite{Gol-Kam} for  flat Lorentz manifolds, i.e. in the case of the homogeneous space $\Mink^{1, n} = \big(\SO(1, n) \ltimes \R^{1+n} \big)/\SO(1, n)$. Thus $\Mink^{1, n} = G/H$ is the Minkowski space of dimension $1+n$, where $G$ is the Poincar\'e group and  $H$ is the Lorentz group. The Auslander conjecture itself  asks for Bieberbach rigidity for compact  unimodular affine manifolds \cite{Aus}. This was generalized more recently in \cite{HKMZ} for Lorentz manifolds modeled on plane waves, which may be viewed as perturbations of proper plane waves. Some related spaces, such as higher-dimensional anti de Sitter spaces $\mathrm{AdS}^{n+1} = \mathrm{SO}(2,n)/\mathrm{SO}(1,n)$, for $n > 2$, are also conjectured (at least in part) to  satisfy Bieberbach rigidity \cite{Zeg}. 
 
On the other hand, there exist many homogeneous spaces where Bieberbach rigidity fails, like the $3$-dimensional anti de Sitter space \cite{Gol, Salein, D-G-K}.

 \subsection{Classical Inoue--Bombieri construction} In this section, we describe $S^0$-Inoue--Bombieri surfaces in a way that illustrates how Bieberbach rigidity can be applied. We first recall the usual description of these manifolds:

\begin{example}[$S^0$-Inoue--Bombieri surfaces] \label{S0 construction}
 Let $A$ be a real $3 \times 3$ matrix in $\GL_3(\Z)$ with two complex conjugate eigenvalues $\alpha$ and $\bar \alpha$ and one real eigenvalue $\lambda > 1$ (take for example the companion matrix of the polynomial $X^3 - X^2 + 3 X -1$). We consider the manifold $\tilde M := \R^3 \times \R^*_+$ on which the group
\begin{equation}
\Gamma := \Z^3 \rtimes \langle \R^3 \times \R^*_+ \ni (X, x) \mapsto (A X, \lambda x) \rangle
\end{equation}
acts, where $\Z^3$ is the canonical lattice in $\R^3$, and $\langle F \rangle$ denotes the group generated by a family $F$. The manifold $M := \Gamma \backslash \tilde M$ is an Inoue--Bombieri surface of type $S^0$. The vector space $\R^3$ decomposes as the direct sum
\[
\ker (A-\alpha)(A-\bar\alpha) \oplus \ker (A-\lambda) \simeq \R^2 \oplus \R,
\]
and there exists a scalar product $b$ on $\R^2$ such that the restriction of $A$ to $\R^2$ is a $b$-similarity of ratio $\vert \alpha \vert = \lambda^{-1/2}$. We denote by $(x, y)$ the coordinates of $\ker (A-\lambda) \times \R^*_+ \simeq \R \times \R^*_+$ and we endow $\tilde M$ with the metric $h = b + y^{-2}(dx^2 + dy^2)$. One has $(\tilde M, h) \simeq (\R^2, b) \times (\R \times \R^*_+, y^{-2}(dx^2 + dy^2)) \simeq \C \times \H^2$, and $\Gamma$ is a subgroup of $\Sim(\C) \times \Isom(\H^2)$.
\end{example}

Surfaces of type $S^0$ arise as quotients of $\C \times \H^2$ by discrete subgroups $\Gamma$ of $ \bf G= \Aut(\C) \times \Aut(\H^2)$ which is notably smaller than the full group $\Aut(\C \times \H^2)$. Here, $\Aut$ denotes the group of holomorphic diffeomorphisms.  Identifying $\C$ with $\R^2$, one has
\[
\Aut(\C) = \Sim^+(\R^2),
\]
where $\Sim^+$ refers to the group of orientation preserving similarities. When $\Gamma \subset \Isom^+(\R^2)\times \Isom(\H^2)$, we are in the purely Riemannian setting. In the complementary case, the projection of $\Gamma$ onto $\Sim^+(\R^2)$ contains essential similarities - i.e. that are not isometries. In this latter case, the resulting quotient is clearly a GIB manifold in the sense of Definition~\ref{GIB def}.

In order to describe such  a group $\Gamma$, one first observes that its projection onto $\Isom(\H^2)$ is not discrete and must be contained in a parabolic subgroup of $\Isom(\H^2)$. Such a group is isomorphic to $\Aff(\R)$ and, up to conjugacy, acts on the upper half plane by $w \mapsto aw + b$, $a \in \R^+, b \in \R$, so that $\Aff(\R) = \R \ltimes \R$. Analogously, $\Aut(\C) \cong \Sim^+(\R^2)$ is a semidirect product $(\SO(2) \times \R) \ltimes \R^2$.  Thus, $\Gamma$ is a subgroup of
\[
G = \big((\SO(2) \times \R) \ltimes \R^2\big) \times (\R \ltimes \R) =  \SO(2) \ltimes   (A \ltimes \R^3),
\] where $A$ is a copy of $\R^2$, acting on $\R^3$ by diagonal matrices $\mathrm{diag} {(e^t, e^t, e^s)}$. 

Assuming Bieberbach rigidity holds, $\Gamma$ should be a lattice in some connected subgroup  $L \subset G$ acting properly on $\C \times \H^2$ (but not isometrically for the product metric). In particular, $L$ must be unimodular. Disregarding the compact factor for the moment, the unique  connected unimodular subgroup of $A \times \R^2$ acting properly and cocompactly is 4-dimensional and  has the form $D \ltimes \R^3$, where $D$ is the line in $A \cong \R^2$ defined by $2t + s = 0$, ensuring that $\mathrm{diag}(e^t, e^t, e^s)$ has determinant $1$.  Consequently,  the largest  subgroup  acting properly cocompactly on $\C \times \H^2$ is the 5-dimensional group ${\bf L} = (\SO(2) \times D) \ltimes \R^3$.   Therefore, finding $\Gamma$ reduces to showing that $\bf L$ indeed possesses lattices, and  
to classify them. It turns out  this is a standard task in the present setting.

This illustrates how Bieberbach rigidity allows one to avoid the delicate dynamical question of properness of the $\Gamma$-action. We believe that this viewpoint of the Inoue--Bombieri examples offers a more conceptual perspective than the one usually given in the literature, which typically emphasizes the explicit arithmetic description of the group $\Gamma$.

Notice that in this situation we obtain a  strong version of Bieberbach rigidity, the connected group $L$ being the same for all the discrete groups $\Gamma$.

\begin{remark}
The full automorphism group $\Aut(\C \times \H^2)$ is huge. It consists of transformations $(z, w) \mapsto (a(w) z + b(w), \phi(w))$, where $\phi \in \Isom^+(\H^2)$, $a: \H^2 \to \C^*, b: \H^2 \to \C$ are holomorphic.
\end{remark}

\subsection{Proof of Theorem~\ref{BRigidity}}  Let $M$ be a GIB manifold. That is, there exists a simply connected Riemannian product $\tilde M := \R^q \times (N, g_N)$ on which a discrete group $\Gamma \leqslant \mathrm{Sim}(\R^q) \times \Isom(N, g_N)$ acts freely, properly and cocompactly, and $M =  \Gamma \backslash \tilde M$. We assume furthermore that $N$ is a symmetric space.

We first observe that, if $N$ has a compact factor - that is $N$ is a Riemannian product $N' \times K$ with $K$ a compact symmetric space -, then $M' := \Gamma \backslash \R^q \times N'$ is still a GIB manifold. Moreover, we may assume that $\Gamma$ acts effectively on $\R^q \times N'$ by replacing it with the quotient by the subgroup acting trivially. If Theorem~\ref{BRigidity} holds for $M'$, there exists a connected group $L'$ acting properly on $M'$ such that $\Gamma'$ is a lattice in $L'$. Consequently, $L = L' \times \Isom(K)^0$ is connected, acts properly on $M$, and has $\Gamma$ as a lattice; this is precisely the group required in Theorem~\ref{BRigidity}. Hence, to prove the theorem, we may assume that $N$ has no compact factor.

We now consider the case where $N$ is of non-compact type; equivalently, $N$ has non-positive curvature and its de Rham decomposition has no Euclidean factor. We will return later to the situation in which $N$ has a Euclidean factor.

Since $(N, g_N)$ is symmetric, hence homogeneous, its isometry group $\Isom(N)$ has finitely many connected components. Replacing $\Gamma$ by a finite-index subgroup if necessary, we may assume that $P \leqslant \Isom(N)^0$, where $P$ is the group introduced in Section~\ref{SectionQuotient}. It follows that $\bar P \leqslant \Isom(N)^0$.

The group $\Isom(N)^0$ is the identity connected component of the isometry group of a symmetric space of non-compact type; in particular, it has a trivial center and it is semisimple. Consequently, $\Isom(N)^0$ is isomorphic to its image in $\GL (\mathfrak{g})$ by the adjoint representation $\mathrm{Ad}$. We denote by $\mathfrak{p}^0$ the Lie algebra of $\bar P^0$.

The image of $\Isom(N)^0$ by $\mathrm{Ad}$ is the identity connected component of the algebraic subgroup $\mathrm{Aut}(\mathfrak{g})$ of $\GL (\mathfrak{g})$, since $\Isom(N)^0$ is semisimple. Let $H$ be the normalizer of ${\bar P}^0$ in $\mathrm{Aut}(\mathfrak{g})$. In particular, $\bar P \leqslant H$, because ${\bar P}^0$ is the identity connected component of $\bar P$. Moreover,
\begin{equation} \label{normalizer}
H = \{A \in \mathrm{Aut}(\mathfrak{g}), \ A (\mathfrak{p}^0) = \mathfrak{p}^0 \}.
\end{equation}
This shows that $H$ is an algebraic subgroup of $\GL (\mathfrak g)$ because it is defined by polynomial equations. Hence $H$ is an algebraic variety over $\R$, and therefore has finitely many connected components. Up to replacing $\Gamma$ by a finite-index subgroup, we may assume that $\bar P$ is contained in $H^0 \leqslant \Isom(N)^0$. In addition, we have the following property:

\begin{lemma} \label{lemmaCocompact}
The group $P$ is cocompact in $H^0$.
\end{lemma}
\begin{proof}
The Riemannian manifold $(N, g_N)$ is a homogeneous space, hence it can be written as a a quotient $G/K$ where $K$ is the isotropy group of an arbitrary point of $N$. In particular, $K$ is compact. Since $P$ acts cocompactly on $N$, the double quotient $P\backslash G /K$ is compact, and therefore $P \backslash G$ is compact as well. Moreover, since $P \subset H^0 \subset G$, it follows that $P \backslash H^0$ is also compact.
\end{proof}

To show that the group we construct acts transitively, we also need the following:

\begin{lemma} \label{transitivity}
The group $H^0$ acts transitively on $N$.
\end{lemma}
\begin{proof}
Since $N$ is a symmetric space of non-compact type, the group $\Isom(N)$ is semisimple, and we have an identification $N \simeq \Isom(N) / K$, where $K$ is the maximal compact subgroup of $\Isom(N)$. The group $\Isom(N)^0$ acts transitively on the compact manifold $\Isom(N)^0 /H$ and $H$ has a finite number of connected components. Therefore, by \cite[Corollary 2]{Mon}, there exists a compact subgroup of $\Isom(N)^0$ acting transitively on $\Isom(N)^0 /H$. In particular, $K$ acts transitively on $\Isom(N)^0 /H$ by maximality. Consequently, $H$ acts transitively on $N$ and so does $H^0$.
\end{proof}

Replacing $\Gamma$ by a finite-index subgroup if necessary, we may assume that the projection of $\Gamma$ onto $\Sim(\R^q)$ preserves orientation. By construction, $H^0$ normalizes the abelian group ${\bar P}^0$. In particular, $H^0$ acts on ${\bar  P}^0$ by conjugation, and this action may be viewed as a matrix group action on the Lie algebra $\mathfrak p^0$ of ${\bar P}^0$. We fix a basis $\mathcal B$ of $\mathfrak p^0$, and we define the map $\varphi$, assigning to any $h \in H^0$ the absolute value of the determinant of $\mathfrak p^0 \ni x \mapsto \mathrm Ad_h x$ in the basis $\mathcal B$. Let $\Sim^+(\R^q)$ denote the group of orientation-preserving similarities of $\R^q$ and let $\rho : \Sim^+(\R^q) \to \R^*_+$ be the map giving the ratio of a similarity of $\R^q$. We consider the subgroup of $\Sim(\R^q) \times \Isom(N)$ defined by
\begin{equation}
L := \{ (s, h) \in \Sim^+(\R^q) \times H^0, \ \rho(s) = \varphi(h)^{-1/q} \}.
\end{equation}

The group $L$ is connected. Indeed, for any $(s, h)$ of $L$, there is a continuous path from $(s, h)$ to $(\rho(s) \mathrm{Id}, h)$ by connectedness of the direct Euclidean group $\R^q \rtimes \SO(q)$. It remains to remark that $H^0$ is connected, and that $\varphi(\cdot)^{-1/q}$ is continuous, to conclude that $H^0 \ni h \mapsto(\varphi(h)^{-1/q}, h)$ has a connected image. We therefore obtain that $L$ is connected. We claim that:

\begin{lemma} \label{le1}
The group $L$ contains $\Gamma$.
\end{lemma}
\begin{proof}
We recall that the group $\Gamma_0 = \Gamma \cap (\R^q \times {\bar P}^0)$, defined in the paragraph following Corollary~\ref{P0abelian} (see also \cite[Lemma 4.18]{Kou}), is a lattice in $\R^q \times {\bar P}^0$. The group $\Gamma$ acts by conjugation on $\R^q \times {\bar P}^0$ and preserves $\Gamma_0$. Thus, its action may be viewed as a matrix group action on the Lie algebra $\R^q \times \mathfrak p^0$ of $\R^q \times {\bar P}^0$. The preimage $\gamma_0$ of $\Gamma_0$ is a lattice in $\R^q \times \mathfrak p^0$, hence $\Gamma$ is a subgroup of $\GL (\Z^{q+m})$ in a basis $\mathcal B_0$ of $\gamma_0$ (which is also a basis of $\R^q \times \mathfrak p^0$). Given $(s, p) \in \Gamma \subset \Sim^+(\R^q) \times G$, the matrix of $\mathrm{Ad}_{(s,m)}$ in $\mathcal B_0$ has determinant $\pm 1$, and so does its matrix in the concatenation $(\mathcal B_q, \mathcal B)$, where $\mathcal B_q$ is the canonical basis of $\R^q$ and $\mathcal B$ is the basis of $\mathfrak p^0$ introduced above. We deduce that the ratio of $s$ is $\varphi(p)^{-1/q}$. Therefore, $(s, p)$ is in $L$.
\end{proof}

Finally, $\Gamma$ and $L$ possess the structural properties required for Bieberbach-type rigidity:

\begin{lemma} \label{le2}
The group $\Gamma$ is a lattice in $L$ and $L$ acts properly on $\tilde M$.
\end{lemma}
\begin{proof}
The group $L$ normalizes $\R^q \times {\bar P}^0$, so $L/(\R^q \times {\bar P}^0)$ is isomorphic to $\SO(q) \times (H^0/{\bar P}^0)$. By Lemma~\ref{lemmaCocompact} $\bar P$ is cocompact in $H^0$, thus $\bar P/{\bar P}^0$ acts cocompactly on $L/(\R^q \times {\bar P}^0)$. It remains to remark that $\Gamma_0$ acts cocompactly on $\R^q \times {\bar P}^0$ to prove that $\Gamma$ acts cocompactly on $L$.

To prove that $L$ acts properly on $\tilde M$, we observe that, since $\R^q \subset L$ acts properly on $\R^q$, is it sufficient to prove that $L/\R^q$ acts properly on $N$. But $L/\R^q$ is isomorphic to $\SO(q) \times H^0$, and this group acts properly on $N$ because $\SO(q)$ is compact and $H^0$ acts properly on $N$ as a closed subgroup of the isometry group of the complete Riemannian manifold $(N, g_N)$.
\end{proof}

Lemma~\ref{transitivity}, Lemma~\ref{le1} and Lemma~\ref{le2} together imply Theorem~\ref{BRigidity} is the situation where $N$ is of non-compact type. When $N$ has a Euclidean factor of dimension $k$, the proof proceeds in the same way provided $k\le2$. Indeed, the adjoint representation of the Euclidean group $E(n)$ is injective, and its image is an algebraic subgroup of $\mathrm{GL}(\mathfrak{e}(n))$, where $\mathfrak{e}(n)$ denotes the Lie algebra of $E(n)$. In the case $k = 1$, we simply enlarge the group $H$ in the proof by including the orientation-preserving Euclidean group $E^+(1)$. Since $E^+(1) \simeq \R$ commutes with $\Isom(N)^0$, the group $H^0$ still normalizes ${\bar P}^0$, and the proof concludes exactly as in the case where $N$ is of non-compact type.

\section{Classification results for $N$ of negative curvature} \label{Section3}

In this section we retain the same setting as in Section~\ref{SectionQuotient}, i.e. $M := \Gamma \backslash (\R^q \times N)$ is a GIB manifold, and we assume that $N$ has negative curvature. Our goal is to understand the constraints this imposes on $N$. We begin with an illustrative example, namely the real hyperbolic space $\H^n$. Understanding this particular case will prove instructive for the analysis of negatively curved manifolds later on.

We recall that a Hadamard manifold $M$ admits a border at infinity, denoted by $\partial M$, defined as the set of geodesic rays quotiented by the relation $\alpha \sim \beta \Leftrightarrow t \mapsto d(\alpha(t), \beta(t))$ is bounded on $[0, + \infty)$. The isometries of $M$ fall into three classes: elliptic isometries, which have a fixed point in $M$, hyperbolic isometries, which fix a complete geodesic in $M$ and act on it by a non-trivial translation, and parabolic isometries, which are all the remaining isometries. 

When the sectional curvature of $M$ is bounded from above by a negative constant, hyperbolic isometries are precisely those that have exactly two fixed points on $\partial M$ and no fixed point in $M$. Likewise, parabolic isometries are those that have exactly one fixed point in $\partial M$ and no fixed point in $M$. For a more complete exposition, see for instance \cite{BGS}.

\subsection{The case of the real hyperbolic space} \label{HyperbolicSpace}

As already noticed, the quotients studied in this paper may be viewed as generalizations of the $S^0$-Inoue--Bombieri surfaces. Using a construction similar to that of Example~\ref{S0 construction}, we can produce examples of GIB manifolds whose universal cover is $\R^q \times \H^n$, where $\H^n$ denotes the $n$-dimensional hyperbolic space.

\begin{example} \label{Ex1}
We consider $\tilde M := \R^{q+n-1} \times \R^*_+$ and choose a matrix $A \in \GL_{q+n-1}(\Z)$ such that there is a decomposition
\[
\R^{n+1} =: E \oplus F, \quad \mathrm{dim}(E) = q,
\]
two scalar products $b_E$ and $b_F$ on $E$ and $F$ respectively, and a real number $\lambda > 0$ satisfying $A\vert_E \in \lambda O(b_E)$ and $A\vert_F \in \lambda^{-q/\mathrm{dim}(F)} O(b_F)$. We let the group
\begin{equation}
\Gamma := \Z^{q+n-1} \rtimes \langle \R^{q+n-1} \times \R^*_+ \ni (x,t) \mapsto (A x, \lambda^{-1} t) \rangle
\end{equation}
act on $\tilde M$. We endow $\tilde M$ with the metric
\begin{equation}
h := b_E + \frac{1}{t^2} (b_F + d t^2).
\end{equation}
One verifies that $(\tilde M, h) \simeq \R^q \times \H^n$. Moreover, the group $\Gamma$ is a subgroup of $\Sim(\R^q) \times \Isom(\H^n)$ and it acts properly, freely and cocompactly on $\tilde M$.

In the case where $q = 2$ and $n = 3$, we may choose the matrix
\begin{equation} \label{Ablock}
A = \left( \begin{matrix} 2 & 1  &  & \\ 1 & 1 & & \\ & & 2 & 1 \\ & & 1 & 1 \end{matrix} \right),
\end{equation}
which is diagonalizable with two eigenvalues $\lambda$ and $\lambda^{-1}$ of multiplicity $2$.
\end{example}

\begin{remark}
From Example~\ref{Ex1} we can formulate an interesting question on matrices with integer coefficients: what are the matrices of $\GL_p(\mathrm{Z})$ (for an integer $p>0$) such that there exists a decomposition $\R^p =: E \oplus F$ satisfying the conditions of the example? A first result in this direction was given in \cite[Proposition 3]{MMP}, showing that when $\dim(E) = 1$, one has $p \in \{2,3\}$. We remark that these matrices induce an Anosov diffeomorphism on the torus with a single Lyapunov exponent.

Note also that we can always construct suitable matrices in the cases $q = n-1$ and $q = 2(n-1)$ taking block matrices as in \eqref{Ablock}.
\end{remark}

The key ingredient in Example~\ref{Ex1} is the existence of the matrix $A$. In fact, we can show:
\begin{claim} \label{matrixExistence}
The existence of such a matrix is a necessary and sufficient condition  for the existence of a group $\Gamma \subset \Sim(\R^q) \times \Isom(\H^n)$ acting freely, properly discontinuously and cocompactly on $\R^q \times \H^n$.
\end{claim}

The sufficiency follows from the construction in Example~\ref{Ex1}, so it remains to establish necessity. Recall that in this setting, the group $\bar P$ is defined in Section~\ref{SectionQuotient} as the closure of the projection of $\Gamma$ onto $\Isom(\H^n)$, and that ${\bar P}^0$ denotes its identity connected component.

We begin by proving a lemma that holds in full generality:

\begin{lemma} \label{P0notcompact}
The group ${\bar P}^0$ is not compact.
\end{lemma}
\begin{proof}
By contradiction, we assume that ${\bar P}^0$ is compact. We already noticed that the group $\Gamma_0 := \Gamma \cap(\R^q \times {\bar P}^0)$ (introduced in the paragraph following Corollary~\ref{P0abelian}) is a lattice in $\R^q \times {\bar P}^0$. By compactness of ${\bar P}^0$, the projection $\Gamma_0'$ of $\Gamma_0$ onto $\R^q$ is discrete because $\Gamma_0$ is discrete. Consequently, $\Gamma_0'$ is a lattice in $\R^q$, and it is preserved by the action of $\Gamma$ by conjugation on $\R^q$. This is possible only if all the similarities in $\Gamma\vert_{\R^q}$ are isometries. Otherwise, we could find non-zero vectors of arbitrarily small size inside $\Gamma_0'$. This is a contradiction  since we assumed that $\Gamma\vert_{\R^q}$ contains at least one similarity of ratio different from $1$.
\end{proof}

\begin{lemma} \label{FixedPoint}
All the elements of $\bar P$ fix a common point of $\partial \H^n$ and ${\bar P}^0$ does not contain hyperbolic isometries.
\end{lemma}
\begin{proof}
We look at the fixed points of the elements of ${\bar P}^0$ and we use the classification of isometries in hyperbolic spaces.

First, assume that there is a parabolic element $p_0$ in ${\bar P}^0$, i.e. $p_0$ fixes exactly one point $x \in \partial \H^n$ on the boundary. Since ${\bar P}^0$ is abelian, any element of ${\bar P}^0$ must fix $x$. Now, pick any element $p \in \bar P$. We know that ${\bar P}^0$ is a normal subgroup of $\bar P$, so $p^{-1} p_0 p$ is in ${\bar P}^0$ and $p^{-1} p_0 p x = x$, thus $p_0 p x = p x$. This implies $p x = x$.

We now assume that there is a hyperbolic element $p_0$ in ${\bar P}^0$, i.e. $p_0$ fixes exactly two points $x_1$ and $x_2$, lying on the boundary $\partial \H^n$. Again, since ${\bar P}^0$ is abelian and connected, any element of ${\bar P}^0$ fixes these two points. Let $p \in \bar P$. We have, as before, $p_0 p x_i = p x_i$ for $i = 1,2$. Up to taking a subgroup of $\Gamma$ of index $2$ in the construction, we can assume that $p x_i = x_i$. Up to conjugation, the group of isometries of $\H^n$ fixing two points of the boundary is $\O(n-1) \rtimes \R^*_+$ when using the model $\H^n \simeq (\R^{n-1} \times \R^*_+, \frac{1}{x_n^2} (dx_1^2 + \ldots + dx_n^2))$. But this group does not act cocompactly on $\H^n$, so this case is impossible.

The last remaining case is when ${\bar P}^0$ contains only elliptic elements, i.e. all elements have a fixed point in $\H^n$. In that situation, each element of ${\bar P}^0$ lies in the isotropy group of a point, which is compact. Since ${\bar P}^0$ is abelian and connected, and therefore the product of a torus with $\R^m$ for some $m$, one has that $m = 0$ and ${\bar P}^0$ is compact, because otherwise some elements would not be contained in a compact subgroup. This case is impossible because of Lemma~\ref{P0notcompact}.
\end{proof}

By Lemma~\ref{FixedPoint}, we may assume that $\bar P$ fixes the point $\infty$ in $\partial \H^n$, using the upper half-space model $\H^n \simeq (\R^{n-1} \times \R^*_+, \frac{1}{x_n^2} (dx_1^2 + \ldots + dx_n^2))$. The isometries of $\Isom({\H^n})$ fixing $\infty$ lie in the group $\Isom(\R^{n-1}) \rtimes \R^*_+$, where $\Isom(\R^{n-1})$ acts on the first factor of $\R^{n-1} \times \R^*_+$, and $\R^*_+$ denotes the group of positive scalar $n \times n$ matrices.

Since$\bar P$ acts cocompactly on $\H^n$, consider the group $\mathcal H := \bar P \cap \Isom(\R^{n-1})$. We then have a decomposition:
\[
\H^n / \bar P \simeq (\H^n / \mathcal H) / (\bar P / \mathcal H).
\]
Moreover, $(\H^n / \mathcal H) \simeq (\R^{n-1} / \mathcal H) \times \R^*_+$, and the quotient $\bar P / \mathcal H$ acts freely on the factor $\R^*_+$. Indeed, if $\bar\gamma \in \bar P / \mathcal H \subset \R^*_+$ has a fixed point in $\R_+^*$, then it must be the identity. Consequently, $\H^n / \bar P$ is compact only if $\R^{n-1} / \mathcal H$ is compact. 

Since $\bar P$ acts cocompactly on $\H^n$, it must contains a non-trivial similarity $p$ of $\R^n$ with ratio $0 < \lambda < 1$. The map $p$ restricts to a similarity $\bar p$ of $\R^{n-1}$ and this restriction acts by conjugation on $\mathcal H$, because $\mathcal H$ is a normal subgroup of $\bar P$. Thus, if $K$ is a compact subset of $\R^{n-1}$ such that $\mathcal H \cdot K = \R^{n-1}$, one has
\[
\mathcal H \cdot \bar p (K) = \bar p \mathcal H \bar p^{-1} \cdot \bar p (K) = \bar p \mathcal H \cdot K = \bar p (\R^{n-1}) = \R^{n-1}.
\]
The diameter of $\bar p (K)$ is $\lambda$ times the diameter of $K$ (for any fixed Euclidean norm on $\R^{n-1}$). Iterating this process shows that we may take the compact $K$ arbitrarily small. It follows that $\mathcal H \cdot y$ is dense in $\R^{n-1}$ for every $y \in \R^{n-1}$. Since the action of $\mathcal H$ is proper, this implies that $\mathcal H \cdot y = \R^{n-1}$.

Thus $\mathcal H$ acts transitively on $\R^{n-1}$. It follows that the connected component of the identity in $\mathcal H$ already acts transitively on $\R^{n-1}$. The group ${\bar P}^0$ is contained in $\Isom(\R^{n-1})$, because ${\bar P}^0$ contains only parabolic elements, and therefore ${\bar P}^0$ is the identity connected component of $\mathcal H = \bar P \cap \Isom(\R^{n-1})$. In particular, ${\bar P}^0$ it acts transitively on $\R^{n-1}$. In order to conclude that ${\bar P}^0 = \R^{n-1}$, we need to prove:

\begin{lemma}
Let $m > 0$ be an integer. The only connected abelian subgroup of $\Isom(\R^m)$ acting transitively on $\R^m$ is the group of translations $\R^m$.
\end{lemma}
\begin{proof}
Let $G$ be such a subgroup of $\Isom(\R^m)$. Any element of $G$ is of the form $x \mapsto R x + t$ where $(R,t) \in \SO(n) \times \R^m$. Since $G$ is abelian, the linear parts of the elements of $G$ commute, so they are all contained in a maximal torus of $\SO(n)$. Thus, up to conjugation, we may assume that the linear parts of the elements of $G$ are contained in the product group $\prod_{i = 1}^{m/2} \O(2)$ if $m$ is even or $\prod_{i = 1}^{m/2} \O(2) \times \{1\}$ if $m$ is odd. Consequently, $G$ is a subgroup of $\prod_{i = 1}^{m/2} E^+(2)$ if $m$ is even or $\prod_{i = 1}^{m/2} E^+(2) \times \R$ if $m$ is odd, where $E^+(2)$ is the group of direct isometries of $\R^2$ and $\R$ acts by translations on $\R$. The projection of $G$ onto each one of the $E^+(2)$ factors is abelian and acts transitively on $\R^2$, because $G$ acts transitively on $\R^m$. It is then sufficient to prove the lemma in the case $m = 2$.

We assume that $m = 2$. Let $I : x \mapsto R x + t$ be an element of $G$. If $R$ is the identity, the isometry is a translation. If $R$ is a non-trivial rotation, then, up to applying a translation, we may assume that $0_{\R^2}$ is a fixed point of $I$, thus $t = 0$. Now, since $G$ acts transitively on $\R^2$, for any $t' \in \R^2$ there exists $R' \in \SO(2)$ such that the map $x \mapsto R' x + t'$ is in $G$. But the isometry $x \mapsto R x$ commute with this map. We obtain $R t' = t'$, and since this is true for any $t' \in \R^2$, the rotation $R$ is the identity.

It follows that $G$ contains only translations, and it contains all of them because it acts transitively on $\R^m$.
\end{proof}

We proved the following:

\begin{proposition}
If $(N, g_N) = \H^n \simeq (\R^{n-1} \times \R^*_+, \frac{1}{x_n^2} (dx_1^2 + \ldots + dx_n^2))$, then, up to conjugation, the group $\bar P$ preserves the point $\infty$ in $\partial \H^n$, and ${\bar P}^0 = \R^{n-1}$.
\end{proposition}

We now assume that $\R^q \times N = \R^q \times \H^n$. We know that the subgroup $\Gamma_0$ of $\Gamma$ defined in Theorem~\ref{Equivalence} is a lattice in $\R^q \times {\bar P}^0 \simeq \R^q \times \R^{n-1}$. Let $\gamma \in \Gamma$ which projects to a non-isometric similarity on $\Sim(\R^q)$. The matrix of the action of $\gamma$ by conjugation of $\R^q \times \R^{n-1}$ is an element $A$ of $\GL_{q+n-1} (\Z)$ when seen in a basis of the lattice $\Gamma_0$. Moreover, $A$ restricts to $E := \R^q$ and to $F := \R^{n-1}$ and is a similarity of ratio $\lambda$ on $E$ and a similarity of ratio $\lambda^{-q/(n-1)}$ on $F$. This proves Claim~\ref{matrixExistence}.

\subsection{The general case of negative curvature} \label{GeneralNegativeCurvature}

In this section, we investigate a more general framework, following an analysis similar to that of the previous section. We assume that the manifold $N^n$ has negative sectional curvature, and since $\bar P$ acts cocompactly on $N$ by isometries, this implies that the curvature is bounded from above by a strictly negative constant.

We can prove the same result as in Lemma~\ref{FixedPoint} for the hyperbolic space:

\begin{lemma} \label{FixedPointGeneral}
All the elements of $\bar P$ fix a common point of $\partial N$ and ${\bar P}^0$ does not contain hyperbolic isometries.
\end{lemma}
\begin{proof}
The proof is the same as for Lemma~\ref{FixedPoint}, except for the part where we show that ${\bar P}^0$ does not contain hyperbolic isometries. Suppose that there exists a hyperbolic element $p_0 \in {\bar P}^0$ fixing two points $x_1$ and $x_2$ on the boundary of $N$. As in the proof of Lemma~\ref{FixedPoint}, we can show that, up to taking a finite-index subgroup, all the elements in $\bar P$ fix $x_1$ and $x_2$, so they preserve the geodesic between these two points. However, the group of isometries that preserves a geodesic in $N$ does not act cocompactly on $N$, because the image of a compact by this group stays at bounded distance from the geodesic.
\end{proof}

By Lemma~\ref{FixedPointGeneral}, the group $\bar P$ fixes a point $x \in \partial N$. Let $H$ be a horosphere centered at $x$ and let $f$ be a Busemann function at $x$. For any $y \in N$, there is a unique geodesic $\gamma_y$ joining $y$ and $x$, and $H \cap \gamma_y$ is reduced to a point denoted by $\eta(y)$. The map $\eta \times f : N \to \R \times H$, $y \mapsto (\eta(y), f(y))$ is a homeomorphism \cite[Proposition 11]{Shiga}. Let $\mathcal H$ be the stabilizer of $H$ in $\bar P$, so $\mathcal H$ contains no hyperbolic isometry. This is a normal subgroup of $\bar P$, thus
\[
N / \bar P \simeq (N/\mathcal H)/(\bar P/\mathcal H).
\]
In addition, it follows from the existence of the homeomorphism $N \simeq \R \times H$, that $N/\mathcal H \simeq \R \times H/\mathcal H$. The group $\bar P / \mathcal H$ acts freely on the factor $\R$, because if $p \in \bar P$ stabilizes a horosphere centered at $x$, it stabilize all the horospheres centered at $x$, and therefore $p$ lies in $\mathcal H$. Since $N/\bar P$ is compact, we deduce that $H/\mathcal H$ is compact.

Following the analysis of the hyperbolic space, we want to prove that $H/\mathcal H$ is a point. Since $\bar P$ acts cocompactly on $N$, the group $\bar P/G$ is non-trivial, and there exists a hyperbolic isometry $p$ in $\bar P$. This isometry preserves the horospheres centered at $x$, and using again the homeomorphism $N \simeq \R \times H$ it can be written as $p =: (p_1, p_2)$ where $p_1 : \R \to \R$ and $p_2 : H \to H$. We know that $H$ is a $C^2$-hypersurface in $N$ and $\mathcal H$ is a subgroup of the isometries of $H$. In order to prove that $H/\mathcal H$ is a point, it is sufficient to prove that $p_2$ contracts the distances and to proceed as in the case of the hyperbolic space. Up to taking $p^{-1}$ instead of $p$, we can assume that $x$ is the repulsive point of the hyperbolic isometry $p$.

\begin{lemma} \label{Contraction}
For any $a \in H$ and $X \in T_a H$, one has $\Vert d_a p_2 (X) \Vert < \Vert X \Vert$.
\end{lemma}
\begin{proof}
Let $a \in H$ and let $X \in T_a H \setminus \{0 \}$. We observe that $p_2$ can be viewed as the composition of $p$ together with the sliding along geodesics joining points of $p(H)$ to the fixed point $x \in \partial N$ until we reach $H$. Consequently, if we take an injective curve $c : (-\epsilon, \epsilon) \to H$ such that $c(0) = a$ and $c'(0) = X$, then $p \circ c$ is a curve on $p(H)$, and sliding along the geodesics induces a map
\begin{equation}
r : (-\epsilon, \epsilon) \times [0, +\infty) \to N, \ (t,u) \mapsto \gamma_{c(t)} (u),
\end{equation}
where $\gamma_y$ is the unit speed geodesic joining $y$ and $x$ for a given $y \in N$. If $\alpha$ is the distance between $p(H)$ and $H$, then $r(\cdot, \alpha)$ is a curve on $H$.

One has $d_a p_2(X) = d_{(0,\alpha)} r (\partial / \partial t)$ and we remark that $J(u) := d_{(0,u)} r (\partial / \partial t)$ is a Jacobi field along the geodesic $\gamma := r(0,\cdot)$. Denoting by $R$ the Riemann curvature tensor of $N$, and since there is a constant $k > 0$ such that the sectional curvature of $N$ is bounded from above by $-k$, one has, following the same computations as in \cite[Lemma IV.2.2]{Bal95}:
\begin{align*}
\Vert J \Vert'' &= \frac{1}{\Vert J \Vert^3} (- \langle R(J, \dot \gamma), \dot \gamma, J \rangle \Vert J \Vert^2 + \Vert J' \Vert^2 \Vert J \Vert^2 - \langle J', J \rangle^2) \\
&\ge - \frac{1}{\Vert J \Vert} \langle R(J, \dot \gamma), \dot \gamma, J \rangle \ge \frac{k}{\Vert J \Vert} (\Vert J \Vert^2 - \langle J, \dot \gamma \rangle^2).
\end{align*}
We know that $J(0) = d_a p(X) \neq 0$ and $J$ never vanishes because the family of geodesics $\gamma_{c(\cdot)}$ meet at the point $x \in \partial N$. In addition, $J$ is never tangent to $\gamma$ because for any $u$, the curve $r(\cdot, u)$ is injective on a horosphere transverse to $\gamma$. We conclude that for any $u \in [0,\alpha]$ one has $\Vert J \Vert (u)  > 0$ and $\Vert J \Vert$ is a strictly convex function.

The geodesics $\gamma_{c(\cdot)}$ all meet at the repulsive point of $p$, namely $x \in \partial N$, thus $\lim_{u \to +\infty} J(u) < + \infty$, and $\Vert J \Vert$ is a strictly decreasing function. Hence, we have
\[
\Vert X \Vert = \Vert d_a p(X) \Vert = \Vert J(0) \Vert > \Vert J(\alpha) \Vert = \Vert d_a p_2(X) \Vert,
\]
which concludes the proof.
\end{proof}

The quotient $H/\mathcal H$ is compact and, by Lemma~\ref{Contraction}, $H$ admits a contraction mapping $p_2$ that acts on $\mathcal H$ by conjugation. Proceeding as in the hyperbolic space case, we choose a compact subset $K$ of $H$ such that $\mathcal H \cdot K = H$. Then
\begin{equation}
\mathcal H \cdot p_2(K) = H,
\end{equation}
so $K$ may be chosen with arbitrarily small diameter. Since $\mathcal H$ acts properly, this implies that $\mathcal H$ acts transitively on $H$. Consequently, the connected component of the identity of $\mathcal H$, namely ${\bar P}^0 \cap \mathcal H = {\bar P}^0$, acts transitively on $H$.

We know that ${\bar P}^0$ is an abelian group of isometries, and $H$ is simply connected. It follows that $H$ is a flat manifold, isometric to $\R^p$ for some $p > 0$, and that it is a smooth submanifold of $N$. The identification $N \simeq \R \times H$ is therefore a diffeomorphism. 

More can be said. Let $\gamma$ be the axis of the hyperbolic isometry $p$. Then ${\bar P}^0 \gamma = N$, and the metric along $\gamma$ determines completely the metric on $N$. The geodesic $\gamma$ is orthogonal to $H$, so the projection $N \to N/{\bar P}^0$ is a Riemannian submersion. Moreover, the same argument as in Section~\ref{HyperbolicSpace} shows that ${\bar P^0}$ is exactly the group of translations of the factor $\R^{n-1}$.

Thus, the Riemannian manifold $N$ is isometric to $(\R \times \R^{n-1}, dt^2 + \langle S_t \cdot, S_t \cdot \rangle)$ where $S_t \in \GL(\R^{n-1})$ and $\langle \cdot, \cdot \rangle$ is the standard scalar product on $\R^{n-1}$. Up to a linear transformation of the factor $\R^{n-1}$, we may assume that $S_0 = \mathrm{Id}$.

We now assume that $N$ is a homogeneous manifold. By a result of Chen \cite[Theorem 4.1]{Chen}, either $N$ is a symmetric space of rank $1$ or $\Isom(N)$ fixes a point of $\partial N$. Assume first that $N$ is not a symmetric space. Then, the entire isometry group fixes $x \in \partial N$, since this point is already the unique fixed point of the group ${\bar P}^0$.

The subgroup of $\Isom(N)$ preserving the axis $\gamma$ is closed because $\gamma$ is a closed subset of $N$. By transitivity of the action of $\Isom(N)$, this subgroup contains uncountably many elements acting freely on $\gamma$; hence, it is non-discrete and its identity component contains a one-parameter subgroup $\varphi_s$ acting by translations along $\gamma$.

Using the diffeomorphism $N \simeq \R \times \R^{n-1}$, under which the geodesic $\gamma$ corresponds to $\R \times \{0\}$, this one-parameter subgroup can be written in the form $\varphi_s : (t, a) \mapsto (t+s, A_s a)$, where $A_s$ is a linear endomorphism of $\R^{n-1}$. Indeed, $\varphi_s$ preserves the horospheres centered at $x$ and permutes the geodesic rays pointing toward $x$. The map $s \mapsto A_s$ is a one-parameter group of invertible matrices, therefore $A_s = e^{- s A}$ for some matrix $A$. Since $\varphi_s$ is a one-parameter group of isometries, for any $X,Y \in T \R^{n-1}$ and any $t \in \R$ we obtain
\begin{align}
\langle S_t X, S_t Y \rangle = \langle S_0 A_{-t} X, S_0 A_{-t} Y \rangle = \langle e^{t A} X, e^{t A} Y \rangle.
\end{align}

It remains to consider the case where $N$ is a symmetric space of rank $1$. Since $N$ is simply connected, it is either the real hyperbolic space $\H^n$, the complex hyperbolic space $\C \H^{n/2}$, the quaternionic hyperbolic space $\H \H^{n/4}$ or the octonionic hyperbolic plane $\C a \H^2$ up to rescaling. We know that the horospheres centered at $x \in \partial N$ are isometric to flat manifolds. The horospheres of $\C \H^{n/2}$ and $\H \H^{n/4}$ are isometric to the Heisenberg group of dimension $n-1$ endowed with a left-invariant metric, and the horospheres of $\C a \H^{2}$ are isometric to a $2$-steps nilpotent Lie group with a left-invariant metric. The only possibility is then $N = \H^n$.

Summarizing, we proved Theorem~\ref{Classification Negative}.

\begin{remark}
In the non-homogeneous case, since there is a hyberbolic isometry $p$ in $\bar P$, the matrices $S_t$ have a periodicity. Indeed, $p =: (p_1, p_2)$ (using the notations introduced above) acts as a linear map $A$ on the factor $\R^{n-1}$ and it is an isometry, so $\langle S_{t + c} A \cdot, S_{t + c} A \cdot \rangle = \langle S_t \cdot , S_t  \cdot \rangle$ where $c$ is the translation part of $p$ along its axis.
\end{remark}

\end{document}